\documentclass{amsart}
\usepackage{amsfonts,amssymb}
\usepackage{graphicx}  
\usepackage{amsmath,amsfonts,amsthm,amssymb,color}
\usepackage[utf8]{inputenc}
\usepackage{lmodern}
\usepackage{natbib}	
\usepackage{url}
\usepackage{systeme}
\usepackage{enumitem}   
\RequirePackage{filecontents}
\usepackage{setspace}
\usepackage{booktabs}
\usepackage{subcaption}
\usepackage{array}
\usepackage{lscape}
\calclayout
\usepackage{todonotes}
\newtheorem{theorem}{Theorem}
\newtheorem{proposition}{Proposition}

\newtheorem{definition}{Definition}

\numberwithin{equation}{section}

\begin{document}
\title[Stability of a Delayed Predator-Prey Model for Puma Concolor]{Stability of a Delayed Predator-Prey Model for Puma Concolor}

\author[Mej\'ias]{Wilson Mej\'ias}

\author[Sep\'ulveda]{Daniel Sep\'ulveda}
\email{w.a.mejias@gmail.com, daniel.sepulveda@utem.cl}
\address{Oficina de Estudios y Políticas Agrarias – ODEPA, Teatinos 40, Santiago, Chile.}

\address{Departamento de Matem\'atica, Facultad de Ciencias Naturales, Matem\'atica y Medio Ambiente, Universidad Tecnol\'ogica Metropolitana, Las Palmeras 3360, \~Nu\~noa, Santiago, Chile.}
\subjclass{}
\thanks{}
\keywords{Puma Concolor, Guase-Type Predator-Prey System, Delay differential systems, Global stability, Absolute stability, Conditional stability}
\subjclass{34K20, 34K60, 92D25, 92D40}
\maketitle
\begin{abstract}
This study presents a mathematical model that describes the relationship between the Puma concolor and its prey using delay differential equations, a Holling type III functional response, logistic growth for the prey, and a Ricker-type function to model intraspecific competition of the pumas. For positive equilibrium, conditions guaranteeing absolute stability are established, based on the delay value and model parameters. The analysis demonstrates the existence of a unique maximal solution for the proposed model, which remains non-negative for nonnegative initial conditions and is well-defined for all $t$ greater than zero. Furthermore, numerical simulations with different parameter values were performed to investigate the effects of systematically removing a percentage of predators or prey. Numerical simulations attempt to exemplify and put into practice the theorems proved in this article.
\end{abstract}

\section{Introduction}
Population dynamics emerges as a scientific discipline dedicated to the comprehensive study of populations, focusing in particular on the mathematical modeling of their behaviors to accurately predict changes, actions, and profound biological consequences \citep{mat_apli}. Moreover, depredation typically involves the scenario of one organism consuming another, either in part or in whole. This dynamic interaction is captured through the prey-predator system, a model that delineates the intricate relationship between two species where one species serves as a feed for the other \citep{smith2007ecologia}. Differential equations are the cornerstone method for projecting population fluctuations over time. Starting from an initial population condition, these equations enable the precise prediction of future population sizes, showcasing their utility in modeling complex interspecific interactions \citep{kitzes2022handbook}.

The native species \textit{Puma concolor} (Puma) is the largest predator in the Chilean mountain ecosystems \citep{toledo2003puma}. Globally, this species is classified as ``Near Threatened", while in Chile it is designated as ``Endangered" \citep{inv_sp}. Among the factors that have a significant impact on this species, notable ones include hunting, urban expansion, change in land use and climate change \citep{rios2009puma}. Furthermore, \textit{ P. concolor} generates a strong conflict with rural communities, who perceive it as one of the main causes of livestock mortality, leading to illegal hunting \citep{bonacic2013ecologia}. Previous studies have estimated puma densities in different areas of Chile, ranging from 0.75 to 2.5 individuals per 100 km² \citep{guarda2017puma}, although there is evidence suggesting that the densities in Chilean Patagonia could reach 3.44 to 6 individuals per 100 km² \citep{bonacic2013ecologia}. This carnivore commonly preys on moose, deer, venison, beavers, squirrels, marmots, mice, rabbits, hares, and wild boars. However, in Chile, it particularly preys on lagomorphs (rabbits) and camelids (Vicuñas and Guanacos) \citep{toledo2003puma}. On the other hand, there are also indications that pumas occasionally consume other carnivores such as smaller felines and foxes \citep{pacheco2004dieta}. Regarding the ecology and social behavior of the species, \citet{rumiz2010roles} indicates that, being a large carnivore, puma populations play a key role in ecosystems by controlling the number of herbivores, thus enabling the regeneration of plant species. Essentially, they serve as population controllers. Furthermore, the distribution of pumas in their natural environment is influenced more by the presence, quantity, and susceptibility of their prey than by their social interactions \citep{rumiz2010roles}. In fact, the research referenced suggests that pumas may relocate to different areas when there is a reduction in prey availability.

Delay differential equations (DDEs) correspond to functional differential equations that have the particularity that their derivative is given in terms of function values at previous moments, making them suitable for modeling and studying biological phenomena \citep{smith2011introduction}. In the literature, it is possible to find applications of delay differential equations in modeling predator-prey interactions \citep{martin2001predator, fan2007permanence, kar2009stability, krisnawan2019model, moussaoui2015impact, garay2020modelo}.

Lately, scientific publications and studies related to the topic of predator-prey modeling have focused on subjects such as the analysis of temporal patterns based on different functional responses \citep{majumdar2022complex, naik2022multiple, jana2022holling, barman2022role} and the analysis of the ``Allee Effect" and the ``Fear Effect" in prey species \citep{li2022population, naik2022complex, devi2022role, lan2022hopf, gokcce2022dynamic}. However, there are no models that integrate delay differential equations, functional responses, and Ricker-type function to model intraspecific competition for the species \textit{P. concolor}. Given that this species is facing conservation issues, it is important to have tools that can simulate different scenarios and quantify the level of impact and response it has to various disturbances. In this context, the need arises to create a model, focusing on its functions according to the predatory characteristics of the Puma.



Section 2 delineates the formulation of the mathematical model characterizing the predator-prey dynamics specific to the species \textit{P. concolor}, encompassing the pertinent factors and their significance. Sections 3.1, 3.2, and 3.3 are dedicated to the rigorous mathematical analysis of the system. Theorem 1 asserts the existence of a unique non-negative solution to the model; Theorem 2 provides sufficient conditions for the extinction of the predator, while Theorems 3 and 4 scrutinize the local stability of the equilibria. Section 3.4 is concerned with the numerical stability analysis of the model's equilibria utilizing a \emph{stability evaluation function} implemented in the R language. Sections 3.5 and 3.6 are allocated to numerical simulations across various parameter values and the assessment of two potential scenarios involving the reduction of prey or predator populations. Section 4 synthesizes the conclusions derived from the investigation.

\section{Mathematical model}
A population dynamics model was developed using a system of delay differential equations, where the predator was associated with a Holling III type functional response, delay in the growth, and intraspecific competition from juvenile to adult stage. This type of model, by incorporating a Ricker-type recruitment function, adds the effect of intraspecific competition between predators to the model and intrinsically captures the carrying capacity of the ecosystem, \citep{gurney1980nicholson}. The growth of the prey population was modeled using a logistic growth equation.


\subsection{Model Formulation}
\noindent

As a starting point, we consider a Gause-type predator-prey system with a time delay $\tau$ in the interaction term of the predator equation, namely:
\begin{equation}\label{Eq: 1}
    \left\{\begin{array}{lcc}
             \dot{x}(t)=x(t)g(x(t))-y(t)p(x(t)),\\
             \\ \dot{y}(t)=-\beta y(t)+ y(t-\tau)p(x(t-\tau)), 
            \end{array}\right.
\end{equation}
where $\tau$ is a discrete delay, $x(t)$ and $y(t)$ denote the population of prey and adult predators at time $t$, respectively. Additionally, we assume that the prey population's dynamics, in the absence of predators, follow the logistic equation. Thus, in system (\ref{Eq: 1}), we consider $g(x)=r(1-x/K)$, where $r$ denotes the maximum growth rate of the prey, and $K$ denotes the carrying capacity of the habitat. 
 
Since the work focuses on the predator species \textit{Puma concolor} and it has been reported that the associated functional response for this species is of Holling type III (see \cite{soriadinamica}), in system (\ref{Eq: 1}), we consider $p(x)=\gamma x^2/(a^2+x^2)$, where $a$ is the half saturation constant and $\gamma$ is the maximum per capita consumption rate of the predator.

Finally, to model intraspecific competition between predators from birth to adulthood, we multiply the interaction term by $\exp(-y/N)$, thus incorporating a Ricker-type recruitment function, where $N$ is the optimal density for predator reproduction, and this expression is evaluated at $t-\tau$. Therefore, we formulate the following Gause-type predator-prey model with delay:
\begin{equation}\label{eq0}
    \left\{\begin{array}{lcc}
             \dot{x}(t)=r\bigg(1-\dfrac{x(t)}{K}\bigg)x(t)-\dfrac{\gamma x(t)^2}{a^2+x(t)^2}y(t),\\
             \\ 
             \dot{y}(t)=-\beta y(t)+ \dfrac{\gamma x(t-\tau)^2}{a^2+x(t-\tau)^2} y(t-\tau)e^{\frac{-y(t-\tau)}{N}}.\\
            \end{array}\right. t \geq 0
\end{equation}


It should be noted that the Ricker-type function incorporates into the model the effect of intraspecific competition among predators and, intrinsically, the carrying capacity of the ecosystem. To facilitate the understanding of the model, we present in Table \ref{tab1} the parameters present in system (\ref{eq0}) and their description.

\begin{table*}[h!]
\small
\centering
\caption{\centering{Simulation parameters in the model (\ref{eq0}) and their ecological definitions}}
    \begin{tabular}{ll}
        \toprule
         Parameter & Ecological Definition\\
        \midrule
        $r$ & Prey reproduction rate \\
        $K$	& Carrying capacity \\
        $a$ & Predator half-saturation constant \\
        $\gamma$ &  Maximum per capita predator consumption rate \\
        $\beta$ &	Predator mortality rate \\
        $N$ & Optimal reproductive size for the predator population \\
        $\tau$ & Period that includes gestation and the time it takes to adulthood for puma. \\
      \bottomrule
    \end{tabular}
    \label{tab1}
\end{table*}

\subsection{A priori model discussion}
\noindent
System (\ref{eq0}) is a model that fits the ecological characteristics of \textit{P. concolor}, considering its predatory habits and intraspecific competition during its developmental stage. However, Berryman \citep{berryman1995credible} defined attributes that allow determining the credibility of a prey-predator model. We refer to:
\begin{enumerate}
    \item The per capita growth rate of prey decreases with predator density, i.e., there is a negative effect of predators on the prey.
    \item The per capita growth rate of predators should increase with prey density, i.e., there is a positive effect of prey on predators.
    \item Predators should have finite appetites and, consequently, maximum reproduction rates.
    \item When food or other resources are scarce relative to population density, the per capita growth rate of predators should decrease with increasing predator density.
\end{enumerate}

Note that the system (\ref{eq0}) satisfies all the above attributes, in fact (1) and (2) directly resemble each other. On the other hand, from the second equation of the system (\ref{eq0}) it follows that predators have a limited reproduction rate, which reaches its maximum at an optimal population size. In addition, for values above this threshold, the reproduction rate decreases. Therefore, (3) and (4) are also verified.

Our model exhibits structural similarities with the framework proposed by \citet{wan2021stability}, which explores the stability of a Gause-type predator-prey fishery model that integrates proportional prey harvesting and a time delay in the Holling type II predator reaction function. In addition, \citet{ruan2009nonlinear} proposed a methodology to study local stability and Hopf bifurcations for predator-prey systems with time delay relying on the second-order transcendental equation associated with the linearization of the system. We emphasize that Ruan's proposal is based on the work of \citet{cooke1982discrete} where a complete analysis of local stability for second-order differential equations with delay, including stability switches, was developed.

Notice that the system (\ref{eq0}) models the species \textit{P. concolor} primarily through the Holling type III function and the appropriate selection of parameters. However, we believe that the model could be further customized by considering variables such as species size.

\section{Results}

This section presents the main findings of this study, where we propose an ecological mathematical model that describes the relationship between the Puma and its prey using Nicholson's blowflies-type and logistic-type models, considering various scenarios of population parameters. 
The existence and uniqueness of solutions of the proposed model were initially examined, followed by the determination of their equilibria and the subsequent analysis of their stability. Additionally, a computational routine was developed in the R programming language to facilitate a numerical stability analysis. The study culminated in an investigation utilizing various numerical simulations and the assessment of two distinct scenarios involving the reduction of either prey or predator populations.

\subsection{Existence and uniqueness of non-negative solutions}

We introduce some definitions and notation for delay differential equations. For $\tau \geq 0$, we consider $\mathcal{C}= C([-\tau,0],\mathbb{R}^2 )$ the Banach space with the norm $||\varphi||_{\tau}=\sup_{-\tau\leq \theta\leq 0}||\varphi(\theta)||$, where $||\cdot||$ is the maximum norm in $\mathbb{R}^2$. Any vector $\mathbf{v}\in \mathbb{R}^2$ is identified in $\mathcal{C}$ with the constant function $v(\theta)=\mathbf{v}$ for $\theta \in [-\tau,0]$. A general system of functional differential equations take the form
\begin{equation}\label{FDE}
\dot{x}(t)=f(t,x_t),    
\end{equation}
where $f:\mathbb{R}\times\mathcal{C} \supset D\mapsto \mathbb{R}^2$  and $x_t$ corresponds to the translation of a function $x(t)$ on the interval $[t-\tau,t]$ to the interval $[-\tau,0],$ more precisely $x_t\in \mathcal{C}$ is given by $x_t(\theta)=x(t+\theta),\; \theta \in [-\tau,0]$.  

A function $x$ is said to be a \textit{solution} of system (\ref{FDE}) on $[-\tau,A)$ if there is $A>0$ such that $x\in C([-\tau,A),\mathbb{R}^2), (t,x_t)\in D$ and $x(t)$ satisfies (\ref{FDE}) for $t\in[0,A)$. For given $\phi\in \mathcal{C}$, we say $x(t;0,\phi)$ is a solution of system (\ref{FDE}) \textit{with initial value} $\phi$ at $0$ if there is an $A>0$ such that $x(t;0,\phi)$ is a solution of equation (\ref{FDE}) on $[0,A)$ and $x_{0}(t;0,\phi)=\phi$.

Since non-negative solutions are significant for population models, the following subsets of $\mathcal{C}$ are often introduced : 
\begin{displaymath}
\mathcal{C}^+:=C([-\tau,0],\mathbb{R}_+^2), \quad \mathcal{C}_0:=\{\phi \in \mathcal{C}^+\, : \, \phi_i(0)>0,\, 1\leq i\leq 2\},
\end{displaymath}
where $\mathbb{R}_+^2=\{\mathbf{x}\in \mathbb{R}^2\,:\, x_i\geq 0,\,  1\leq i\leq 2\}$.
\begin{theorem} \label{Theo1}
System (\ref{eq0}) has a unique non-negative solution defined over $[-\tau,+\infty)$ for each initial condition $\phi=(\phi_1,\phi_2)^T\in \mathcal{C}^+$.
\end{theorem}

\begin{proof}
We will denote by $f(t,x(t),x(t-\tau))$ the right hand side of system (\ref{eq0}) and $x(t)=  (x_1(t), x_2(t))^T$, then (\ref{eq0}) can be written as,
\begin{equation}\label{Semi-Linear}
\dot{x}(t)=f(x(t),x(t-\tau)),  
\end{equation}
where  $f: \mathbb{R}^2_+\times \mathbb{R}^2_+   \to\mathbb{R}^2$ is given by
\begin{equation}
\label{eq2}
    f(\mathbf{x}, \mathbf{y})=
    \begin{pmatrix}
    f_1(\mathbf{x}, \mathbf{y})\\f_2(\mathbf{x}, \mathbf{y})
    \end{pmatrix} 
    =
    \begin{pmatrix}
        r\left(1-\dfrac{x_{1}}{K}\right)x_{1}-\dfrac{\gamma x_{1}^2}{a^2+x_{1}^2}x_{2}\\ 
        -\beta x_{2}+ \dfrac{\gamma y_{1}^2}{a^2+y_{1}^2} y_{2}e^{\frac{-y_{2}}{N}}
    \end{pmatrix},
\end{equation}
with $\mathbf{x}=(x_1,x_2)^{T}$ and $\mathbf{y}=(y_1,y_2)^{T}$.
We denote $f_{\mathbf{x}}$ to the derivative of $f$ respect to the variable $\mathbf{x}$, consequently the map $f_{\mathbf{x}}: \mathbb{R}^2_+\times \mathbb{R}^2_+\to M(\mathbb{R})_{2\times 2}$  defined by 
$$f_{\mathbf{x}}(\mathbf{x}, \mathbf{y})=\begin{pmatrix} f_1/\partial x_1& f_1/\partial x_2 \\ 
                       f_2/\partial x_1& f_2/\partial x_2
\end{pmatrix}=\begin{pmatrix}
        \frac{r}{K}(K-2x_{1})-\frac{2\gamma a^2 x_{1}x_{2}}{(a^2 + x_{1}^2)^2} & \frac{\gamma x_{1}^2}{a^2 + x_{1}^2}\\ 
        0 & -\beta
    \end{pmatrix}$$
is continuous over $\mathbb{R}^2_+\times \mathbb{R}^2_+$. Now, applying Theorems 3.1 and 3.2 in \cite{smith2011introduction}, it follows that the system (\ref{eq0}) has a unique solution defined over a maximal interval for each initial condition $\phi\in \mathcal{C}^+$. In order to show that $x(t;0,\phi)$ takes non-negative values, we note that if $\mathbf{x}, \mathbf{y}\in \mathbb{R}^2_{+}$ and $x_1=0$, then 
\begin{displaymath}
    f_{1} (\mathbf{x}, \mathbf{y}) =  r\left( 1-\frac{x_{1}}{K}\right) x_{1}-\dfrac{\gamma x_{1}^2}{a^2+x_{1}^2}x_2 = 0.
\end{displaymath}
On the other hand, for $\mathbf{x}, \mathbf{y}\in \mathbb{R}^2_{+}$ and $x_2=0$,  it follows that
\begin{displaymath}
    f_{2} (\mathbf{x}, \mathbf{y}) =  -\beta x_{2}+ \dfrac{\gamma y_{1}^2}{a^2+y_{1}^2} y_{2}e^{\frac{-y_{2}}{N}} = \dfrac{\gamma y_{1}^2}{a^2+y_{1}^2} y_{2}e^{\frac{-y_{2}}{N}}\geq 0.
\end{displaymath}

Consequently, each non-negative initial condition $\phi$ has a corresponding solution $x(t;0,\phi)$ that takes non-negative values for $t$ in the maximal interval. 

We claim that the solutions of (\ref{eq0}), corresponding to non-negative initial conditions, are defined for all $t \geq 0$. Otherwise, they would be defined over an interval $[-\tau, A)$, where $0 < A < \infty$. Since $x(t)$ is a solution of (\ref{eq0}), it follows that  
$x_1(t)$ satisfies:
\begin{displaymath}
    \dot{x}_1=r\bigg(1-\dfrac{x_1}{K}\bigg)x_1-\dfrac{\gamma x_1^2}{a^2+x_1^2}x_2.
\end{displaymath}
Given that $x_1$ and $x_2$ are non-negative, it follows that
\begin{displaymath}
    \dot{x}_1 \leq rx_1,
\end{displaymath}
or equivalently
\begin{displaymath}
    x_1(t)\leq \phi_1(0) e^{rt}.
\end{displaymath}
Therefore we have:
\begin{equation}
\label{lim1}
    \lim_{t \to A^-} x_1(t)\leq \lim_{t \to A^-} \phi_1(0) e^{rt} = \phi_1(0) e^{rA } < +\infty.
\end{equation}
In a similar way, the second equation of system (\ref{eq0}) corresponds to:
\begin{displaymath}
    \dot{x}_2=-\beta x_2+ \dfrac{\gamma x_{1_t}^2}{a^2+x_{1_t}^2}  x_{2_t} e^{\frac{-x_{2_t}}{N}},
\end{displaymath}
by using the fact that $x_1$ and $x_2$ take non-negative values, it follows that  $|e^{\frac{-x_{2_t}}{N}}|\leq 1$, and consequently we obtain: 
\begin{displaymath}
    \dot{x}_2 \leq \dfrac{\gamma x_{1_t}^2 x_{2_t}}{a^2+x_{1_t}^2}.  
\end{displaymath}
Next,  since the functional response is bounded from above we have:
\begin{displaymath}
    \dfrac{dx_2}{dt} \leq \gamma x_2(t-\tau).  
\end{displaymath} 
By using the method of steps we deduce that 
\begin{displaymath}
   x_2(t) \leq  ||\phi_2|| \frac{\gamma^n(t-n\tau)^n}{n!},
\end{displaymath}
for $t \in [n\tau, (n+1)\tau)$. Consequently, making $t\to A$ it follows that
\begin{equation}
\label{lim2}
    \lim_{t \to A^-} x_2(t)\leq \lim_{t \to A^-}  ||\phi_2|| \frac{\gamma^n(A-n\tau)^n}{n!} < +\infty.
\end{equation}

This estimate ensures that $A=+\infty$, because if $A < +\infty$, then $|x(t)| \to \infty$ as $t \to A $, contradicting the estimates (see, for instance, Theorem 3.2 \cite{smith2011introduction}).

\end{proof}

\subsection{Equilibrium points and Linearization of (\ref{eq0}) }
To determine the equilibrium points of system (\ref{eq0}), the differential equations that comprise it were set to zero, assuming that at equilibrium, the population values with and without delays are identical, that is, $x = x_{t}$ and $y = y_{t}$. This yields the following system of equations:
\begin{equation}
\begin{array}{rcl}
  0&=&\displaystyle  r\left( 1-\frac{x}{K}\right) x-\frac{\gamma x^2}{a^2+x^2}y, \\ \\
  0&=& \displaystyle -\beta y + \frac{\gamma x^2}{a^2+x^2} y e^{\frac{-y}{N}}. 
  \end{array}
  \label{eq_equilibrio}
\end{equation}
It is not difficult to determine that the points $E_1=(0,0)$ and $E_2=(K,0)$ are solutions of system (\ref{eq_equilibrio}), thus equilibria of system (\ref{eq0}). However, it is particularly important to determine the existence of equilibria with both positive coordinates. By searching for positive solutions of (\ref{eq_equilibrio}), the following system of equations is obtained:
\begin{equation}
\label{3.7}
    \begin{array}{rcl}
         y&=& \displaystyle  \frac{r}{\gamma}\frac{(1-x/K)(a^2+x^2)}{ x},  \\ \\
         y&=& \displaystyle N\ln\left(\frac{\gamma}{\beta}\frac{x^2}{a^2+x^2}\right) . 
    \end{array}
\end{equation}
Note that from the first expression in (\ref{3.7}), it follows that $y>0$ for $x$ in the interval $(0, K)$ and $y\leq 0$ for $x$ in the interval $[K,+\infty)$. Thus, we must look for solutions for values of $x\in (0,K)$. On the other hand, the second expression in (\ref{3.7}) will take positive values in the desired interval if and only if the following condition is satisfied:
\begin{equation}
\label{3.8}
\frac{\gamma}{\beta}\frac{K^2}{a^2+K^2}>1.
\end{equation}

In other words, condition (\ref{3.8}) is necessary and sufficient to guarantee the existence of at least one positive solution of system (\ref{eq_equilibrio}).

It is recalled that the linearization of system (\ref{eq0}) around a stationary solution $\mathbf{u}=(u_1,u_2)$ takes the form:
\begin{displaymath}
    x'(t)=Ax(t)+Bx(t-\tau)
\end{displaymath} and the associated characteristic equation corresponds to
\begin{displaymath}
 \det(\lambda I - A - e^{-\lambda \tau }B) = 0;
\end{displaymath}
where the matrices $A$ and $B$ are defined as:
\begin{displaymath}
    A=
\begin{pmatrix}
     \frac{\partial f_1}{\partial x} & \frac{\partial f_1}{\partial y} \\ \\ 
     \frac{\partial f_2}{\partial x} & \frac{\partial f_2}{\partial y} 
\end{pmatrix}=
    \begin{pmatrix}
        a_{11} & a_{12}  \\ 
         0 &  a_{22}
    \end{pmatrix}, 
\quad
\textnormal{and}
\quad
B=
\begin{pmatrix}
     \frac{\partial f_1}{\partial x_t} & \frac{\partial f_1}{\partial y_t} \\ \\
     \frac{\partial f_2}{\partial x_t} & \frac{\partial f_2}{\partial y_t} 
\end{pmatrix}=
    \begin{pmatrix}
     0  &0 \\ 
     b_{21} 
 & b_{22} 
    \end{pmatrix},    
\end{displaymath}
with
\begin{displaymath}
    \begin{array}{c}
\displaystyle a_{11}=  r\left(1-\frac{2u_1}{K}\right) - \frac{2\gamma u_1 u_2a^2}{(a^2+u_1^2)^2},\quad a_{12}= \frac{\gamma u_1^2}{a^2+u_1^2} ,\quad a_{22}=-\beta,\\ \\
\displaystyle
 b_{21}= \frac{2\gamma u_1 u_2 a^2}{(a^2+u_1^2)^2}e ^{\frac{-u_2}{N}},\quad \textnormal{and}\quad b_{22}=\frac{\gamma u_1^2 }{a^2+u_1^2} \left(1-\frac{u_2}{N}\right)e ^{\frac{-u_2}{N}}.
    \end{array}
\end{displaymath}
Thus, the characteristic equation is equivalent to:
\begin{equation}
\label{2.6}
\lambda^2 -(a_{11}+a_{22})\lambda +a_{11}a_{22} - e^{-\lambda \tau} \left(b_{22}\lambda -b_{22}a_{11} +a_{12}b_{21}\right) =0.
\end{equation}

\begin{definition} \label{Def1}
The steady state $\mathbf{u}$ of system (\ref{eq0}) is said to be \textit{absolutely stable} (i.e., asymptotically stable regardless of the delays) if it is asymptotically stable for all delays $\tau \geq 0$. $\mathbf{u}$ is called \textit{conditionally stable} (i.e., asymptotically stable depending on the delays) if it is asymptotically stable for $\tau$ within some interval, but not necessarily for all delays $\tau \geq 0$.
\end{definition}

\subsubsection{Transcendental polynomial equation of second degree}

It is emphasized that for the discrete delay predator-prey system (\ref{eq0}), the characteristic equation of the linearized system at a non-zero steady state corresponds to a transcendental second-degree polynomial equation, which has been studied by various researchers \citep{cooke1982discrete,brauer1987absolute,ruan2001absolute, ruan2009nonlinear}. Next, we state a result from \citet{cooke1982discrete} regarding the distribution of the roots of equation (\ref{2.6}), and for this purpose, we introduce some equations and notation, starting with the general transcendental second degree equation:
\begin{equation}
\label{ET2}
F(\lambda):=\lambda^2 +m\lambda +n\lambda e^{-\lambda \tau}+p+qe^{-\lambda \tau}=0.
\end{equation}
Note that for $\tau=0$ the above  equation becomes into
\begin{equation}
\label{E2}
\lambda^2 +(m+n)\lambda +p +q=0.
\end{equation}
Suppose that both roots of equation (\ref{E2}) have negative real part, the Routh-Hurwitz criteria (pp. 435, H. R. Thieme, 2003)  establish that if
$$m+n>0,\quad \textnormal{and}\quad p+q>0,$$
then the roots of (\ref{E2}) have negative real part. From equation (\ref{ET2}) we have:
\begin{equation}\label{ET3}
 e^{-\lambda \tau}=-\frac{\lambda^2 +m\lambda +p}{n\lambda +q}.
\end{equation}

Note that if equation (\ref{ET2}) have any purely imaginary solution, namely $\lambda= i\omega$, then it follows:
$$
 1=\left|\frac{\omega^2 -im\omega -p}{in\omega +q}\right|
$$
which is equivalent to
\begin{equation}\label{Ec4}
 \omega^4+( m^2-n^2-2p )\omega^2+p^2-q^2=0,
\end{equation}
and its roots $\omega_{\pm}^2$ are given by
\begin{equation}
\label{Roots}
\omega_{\pm}^2=\frac{1}{2}(2p+n^2-m^2) \pm \frac{1}{2}\left[ (2p+n^2-m^2)^2 -4(p^2-q^2)\right]^{1/2}.
\end{equation}
Consequently, we have:
\begin{enumerate}
\item[a)] If $2p+n^2-m^2<0$ and $p^2-q^2>0$ hold, then the equation (\ref{Ec4}) has no positive solutions. Which is equivalent to the characteristic equation (\ref{ET2}) has no purely imaginary roots.

\item[b)] If $p^2-q^2<0$ holds, then the equation (\ref{Ec4}) has one positive solution. Which is equivalent to the characteristic equation (\ref{ET2}) has one purely imaginary root, $\lambda=i\omega_+$, $\omega_+>0$.

\item[c)] Assume that 
$$p^2-q^2>0,\, n^2-m^2+2p>0,\quad \mbox{and}\quad (n^2-m^2+2p)^2>4(p^2-q^2),$$
are satisfied, then the characteristic equation (\ref{ET2}) has two purely imaginary roots, $\lambda_{\pm}=i\omega_{\pm}$, $\omega_+>\omega_->0$.
\end{enumerate}

In order to study the sign of the derivative of $\Re{\lambda}$ respect to  $\tau$ for
$\lambda$ purely imaginary, we differentiate equation (\ref{ET2})  to obtain
\begin{equation}
\label{DET2}
\underbrace{\left(2\lambda  +m +\left[n - \tau \left( n\lambda + q\right) \right)e^{-\lambda \tau}\right)}_{\frac{dF}{d\lambda}}\frac{d\lambda}{d\tau}=\lambda \left(n\lambda +q\right) e^{-\lambda \tau}.
\end{equation}
Note that the purely imaginary roots are simple. In other case, we have simultaneously $F(\lambda)=0$ and $dF(\lambda)/d\lambda=0$ for some $\lambda=i\omega$. From equation (\ref{DET2})  it follows that $dF(\lambda)/d\lambda=0$ implies $\lambda \left(n\lambda +q\right) e^{-\lambda \tau}=0$, or 
$$0=|\lambda \left(n\lambda +q\right) e^{-\lambda \tau} |=|n\lambda^2 +q\lambda |,$$
which is equivalent to $n=0=q$.  In this case $dF(\lambda)/d\lambda=0$ evaluate at $i\omega$ become in $2i\omega+m=0$,  which is satisfied only for $m=0=\omega$. Thus, a necessary condition for the existence of roots of the order of greater than one for the characteristic equation (\ref{ET2})  is that the coefficients $m,n,q$ be zero. This contradicts our assumption $m+n>0$. We conclude that all the purely imaginary roots of the characteristic equation (\ref{ET2}) are simple. Consequently, we can compute $d\lambda/d\tau$ using an implicit derivative.

Note that from equations (\ref{ET3}), (\ref{Ec4}) and (\ref{DET2}) we have
\begin{displaymath}
\begin{array}{c}
\displaystyle e^{\lambda \tau}=-\frac{n\lambda +q}{\lambda^2+m\lambda +p}, \quad m^2\omega^2+(\omega^2-p)^2=q^2+n^2\omega^2,\\ \\ \displaystyle \left( \frac{d\lambda}{d\tau}  \right)^{-1}=\frac{(2\lambda +m)e^{\lambda \tau}+n}{\lambda(n\lambda +q)}-\frac{\tau}{\lambda}.
\end{array}
\end{displaymath}
Next, by using the above indentities, we can determine the sign of $\frac{d(\Re \lambda)}{d\tau}$,
\begin{displaymath}
\begin{array}{rcl}
\mathrm{sign}\left\{\frac{d(\Re \lambda)}{d\tau} \right\}_{\lambda=i\omega} &=& \mathrm{sign}\left\{\Re \left(\frac{d\lambda}{d\tau} \right)^{-1}\right\}_{\lambda=i\omega}  \\ \\
&=& \mathrm{sign}\left\{\Re \left(\frac{(2\lambda +m)e^{\lambda \tau}+n}{\lambda(n\lambda +q)}-\frac{\tau}{\lambda}\right)\right\}_{\lambda=i\omega}  \\ \\
&=& \mathrm{sign}\left\{\Re \left(-\frac{(2\lambda +m)}{\lambda(n\lambda +q)}\frac{n\lambda +q}{\lambda^2+m\lambda +p}\right)+\Re \left(\frac{n}{\lambda(n\lambda +q)}\right)\right\}_{\lambda=i\omega}  \\ \\
&=& \mathrm{sign}\left\{\Re \left(-\frac{(2\lambda +m)}{\lambda(\lambda^2+m\lambda +p)}\right)+\Re \left(\frac{n}{\lambda(n\lambda +q)}\right)\right\}_{\lambda=i\omega}  \\ \\
&=& \mathrm{sign}\left\{\Re \left(-\frac{(2i\omega +m)}{i\omega(-\omega^2+m i\omega +p)}\right)+\Re \left(\frac{n}{i\omega(n i\omega+q)}\right)\right\}  \\ \\
&=& \mathrm{sign}\left\{\Re \left(-\frac{(2i\omega +m)}{1}\frac{(i\omega^3-m \omega^2 -p i\omega)}{m^2\omega^4 +\omega^2(\omega^2-p)^2}\right)+\Re \left(\frac{n}{1}\frac{(-n\omega^2- i q\omega)}{n^2\omega^4+  q^2\omega^2}\right)\right\}  \\ \\
&=& \mathrm{sign}\left\{\frac{m^2-2(p-\omega^2)}{m^2\omega^2+(\omega^2-p)^2}-\frac{n^2}{n^2\omega^2+q^2}\right\} \\ \\
&=& \mathrm{sign}\left\{\frac{m^2-2(p-\omega^2)}{n^2\omega^2+q^2}-\frac{n^2}{n^2\omega^2+q^2}\right\} =\mathrm{sign}\left\{m^2-n^2-2p+2\omega^2\right\} .
\end{array}
\end{displaymath}

Changing $\omega^2$ by $\omega_{\pm}^2$ in the previous equation, it follows that the sign is positive for $\omega^2_{+}$ and negative for $\omega^2_{-}$. For case b), there is only one positive purely imaginary, $\lambda=i\omega_+$. Consequently, only one cross of the imaginary axis from left to right is possible as $\tau$ increases (since the real part increases with respect to $\tau$), and the stability of the zero solution can only be a loss. For case c), cross from left to right of the imaginary axis as $\tau$ increases occurs for any $\tau$ corresponding to a value $\omega_+$; on the other side, cross from right to left occurs for any $\tau$ corresponding to $\omega_-$.

Let us suppose that equation (\ref{ET2}) has at least one purely imaginary solution, and from (\ref{ET3}) we have:

\begin{displaymath}
 e^{-i\omega \tau}=\displaystyle\frac{\omega^2 -i m\omega -p}{i n\omega +q}=\frac{q\omega^2-pq-nm\omega^2+(np\omega-qm\omega-n\omega^3)i}{q^2+n^2\omega^2}.   
\end{displaymath}
Since the real part of both number must be equal we obtain:
\begin{equation}\label{tau}
\cos(\omega \tau)=\frac{q\omega^2 -pq-nm\omega^2}{q^2 + n^2\omega^2},
\end{equation}
hence the values of $\tau$ for which there are imaginary roots are
\begin{equation}
\label{2.11}
\tau_{n}^{\pm}=\frac{1}{\omega_{\pm}}\arccos\left\{ \frac{q(\omega^2_{\pm}-p)-mn\omega^2_{\pm}}{q^2+n^2\omega^2_{\pm}}  \right\}+\frac{2\pi j}{\omega_{\pm}},\quad j=0,1,2,\cdots 
\end{equation}

The following hypotheses are introduced to facilitate the statement of the result of \cite{cooke1982discrete} concerning the distribution of the characteristic roots of equation (\ref{2.6}), and thus determine the local stability of an equilibrium of (\ref{eq0}).

\begin{enumerate}
    \item[(A1)] $m+n>0$,
    \item[(A2)] $q+p>0$,
    \item[(A3)] $n^2-m^2+2p<0$ and $p^2-q^2>0$,
    \item[(A4)] $p^2-q^2<0$,
    \item[(A5)] $p^2-q^2>0, n^2-m^2+2p>0,$ and $(n^2-m^2+2p)^2>4(p^2-q^2)$.
\end{enumerate}

\begin{proposition}\cite[pp. 605-607]{cooke1982discrete}\label{prop:C-G}
\label{Prop1}
Let $\tau_j^{\pm}$ ($j=0,1,2,\cdots $) be defined by (\ref{2.11}).

\begin{enumerate}
\item[i)] If (A1)-(A3) hold, then all the roots of equation (\ref{2.6}) have negative real parts for every $\tau \geq 0$.

\item[ii)] If (A1), (A2), and (A4) hold, then when $\tau \in [0,\tau^+_0)$, all the roots of equation (\ref{2.6}) have negative real parts. When $\tau=\tau^+_0$, equation (\ref{2.6}) has a pair of purely imaginary roots $\pm i\, \omega_+$. And when $\tau>\tau^+_0$, equation (\ref{2.6}) has at least one root with a positive real part.

\item[iii)] If (A1), (A2), and (A5) hold, then there exists a positive integer $k$ such that there are $k$ changes from stability to instability to stability. In other words, when
\[\tau\in [0,\tau_0^+), (\tau_0^- ,\tau_1^+),\cdots, (\tau_{k-1}^- ,\tau_k^+), \] all roots of equation (\ref{2.6}) have negative real parts, and when
\[\tau\in [\tau_0^+,\tau_0^-), [\tau_1^+,\tau_1^-),\cdots, [\tau_{k-1}^+,\tau_{k-1}^-), \textnormal{ y }\tau>\tau^+_k, \] 
the equation (\ref{2.6}) possesses at least one root with a positive real part.
\end{enumerate}
\end{proposition}
\subsection{Stability analysis of the model}
Recall that the points $E_1=(0,0)$ and $E_2=(K,0)$  are equilibria of system (\ref{eq0}), furthermore if (\ref{3.8}) holds, then there exists at least one positive equilibrium $E_3=(x_*,y_*)$. 

By straightforward calculations it is possible to deduce that the equilibrium $E_1$ is locally unstable, so we focus on the equilibria points $E_2$ and $E_3$.

\subsubsection{Stability of equilibrium $(K,0)$}

We begin by considering the case where the system (\ref{eq0}) has only equilibria $E_1$ and $E_2$. In particular, when $\beta>\gamma$ the condition (\ref{3.8}) is not satisfied and we have the following global stability result of $E_2$.

\begin{theorem}
\label{Theo3}
    If $\beta >\gamma$, then any solution $x(t,\phi)$ such that $\phi > 0$ converges to $(K,0)$ as $t$ tends to infinity. 
\end{theorem}

\begin{proof}
   Consider the Liapunov functional defined by $V: \mathcal{C}_0 \rightarrow \mathbb{R}$ given by
\begin{equation}
    V(\phi) = \phi_2(0) + \gamma \int_{-\tau}^{0} \frac{\phi_1(\theta)^2}{a^2 + \phi_1(\theta)^2} \phi_2(\theta) e^{-\frac{\phi_2(\theta)}{N}} d\theta.
\end{equation}
Note that if $x(t,\phi)$ is a solution of system (\ref{eq0}) with $\phi \in \mathcal{C}_0$, then $x_t(\phi) \in \mathcal{C}_0$ for $t \geq 0$, and
\begin{equation}
    V(x_t) = y(t) + \gamma \int_{-\tau}^{0} \frac{x(t+\theta)^2}{a^2 + x(t+\theta)^2} y(t+\theta) e^{-\frac{y(t+\theta)}{N}} d\theta.
\end{equation}
Observing that $V(x_t)$ is a function of $t \in [0,\infty)$, it is differentiable, and
\begin{equation}
    \dot{V}(x_t) = \left(\gamma \frac{x^2(t)}{a^2+x^2(t)}e^{-\frac{y(t)}{N}} - \beta\right) y(t) \leq \left(\gamma - \beta\right) y(t).
\end{equation}
Since $\beta > \gamma$, it follows that $t \mapsto V(x_t(\phi))$ decreases on $[0,+\infty)$. Assuming $\phi \neq (K,0)$, we have $V(\phi) > 0$. Next, we define $G = \{ \psi \in \mathcal{C}_0: V(\psi) \leq V(\phi)\}$. Then $G$ is closed, positively invariant, and contains $x_t(\phi)$ for all $t \geq 0$. Our calculations above, with $\phi = \psi \in G$ and $t = 0$, show that $\dot{V}(\phi) \leq 0$. Furthermore, Theorem \ref{Theo1} guarantees that every orbit of (\ref{eq0}) is bounded. By the LaSalle invariance principle (see, for instance, Theorem 5.17 \cite{smith2011introduction}), $\omega(\phi)$ is a compact nonempty subset of the maximal invariant subset $S = \{\psi \in G: \dot{V}(\psi) = 0\} = \{\psi \in G: \psi(0) = (K,0)\}$. However, any invariant subset $A$ of $S$ must satisfy $A = \{(K,0)\}$ because for each $\psi \in A$, the corresponding solution $x(t) := x(t,\psi)$ satisfies $x_t \in A \subset S$ for $t \geq 0$, which implies $x_{t}(-\tau) = (K,0)$. For $0 \leq t \leq \tau$, this implies $\psi = (K,0)$.
\end{proof}

In addition, for $\beta \leq \gamma$,  we have the following local stability result of $E_2$.

\begin{theorem}
\label{Theo2}
Assume that
\begin{equation}\tag{H1}\label{H1}
\beta>\frac{\gamma  K^2}{a^2+K^2},
\end{equation}
then the equilibrium $(K,0)$  of system (\ref{eq0})  is locally absolutely stable (locally asymptotically stable for all $\tau \geq 0$).



\end{theorem}
\begin{proof}
    We linealize system (\ref{eq0}) around the equilibrium $(K,0)$, and the characteristic equation (\ref{2.6}) becomes:
\begin{displaymath}
\lambda^2+\underbrace{\left(r+\beta \right)}_{m}\lambda + \underbrace{\left(-\frac{\gamma K^2}{a^2+K^2}\right)}_{n}\lambda e^{-\lambda \tau} +\underbrace{r\beta}_{p}+\underbrace{\left(-\frac{\gamma r K^2}{a^2+K^2}\right)}_{q}e^{-\lambda \tau} =0.
\end{displaymath}
From assumption (\ref{H1})  we have 
\begin{displaymath}
r+\beta-\frac{\gamma K^2}{a^2+K^2}>0,\quad \textnormal{and}\quad r\left(\beta-\frac{\gamma  K^2}{a^2+K^2}\right)>0,   
\end{displaymath}
hence $m+n>0$ and $p+q>0$ hold. Furthermore, since $\beta+\frac{\gamma  K^2}{a^2+K^2}$ is positive, it follows that:
\begin{displaymath}
    \beta^2-\left(\frac{\gamma  K^2}{a^2+K^2}\right)^2>0,
\end{displaymath}
whence 
\begin{displaymath}
    2r\beta+\frac{\gamma^2 K^4}{(a^2+K^2)^2}-(r+\beta)^2<0, \quad  \textnormal{and}\quad r^2\beta^2-\frac{r^2\gamma^2 K^4}{(a^2+K^2)^2}>0.
\end{displaymath}
Apply Proposition \ref{prop:C-G} i), and the result is obtained.  
\end{proof}

\subsubsection{Local stability of equilibrium $(x_*,y_*)$}
Let us recall that a positive equilibrium can exist if and only if the condition (\ref{3.8}) is satisfied, or equivalently 
$$
\gamma\frac{ K^2}{a^2+K^2}>\beta.
$$
In order to ease the analysis of local stability of the positive equilibrium $(x_*,y_*)$ we introduce the following constants:

\begin{equation}\label{mnpq}
    \begin{array}{c}
\displaystyle m = r\left(1-\frac{x_*}{K}\right)  \left(\frac{2a^2}{a^2+x_*^2}-1\right) +\beta +\frac{rx_*}{K},\quad 
 n= \beta \left(\frac{y_*}{N}-1\right),\\ \\
 \displaystyle p = -\beta(\beta-m),\quad
  q= -\beta r\left(1-\frac{2x_*}{K}\right).
    \end{array}
\end{equation}
\begin{theorem}
\label{Theo4}
Let assume that (\ref{3.8}) holds, if one of the positive equilibria $E_3=(x_*,y_*)$ verifies:
\begin{equation}\tag{H2}\label{H2}
   \frac{\frac{rx_*}{K} +\beta \frac{y_*}{N}}{r\left(1-\frac{x_*}{K}\right)}>\left(1-  \frac{2a^2}{a^2+x_*^2}\right), 
\end{equation}
and
\begin{equation}\tag{H3}\label{H3}
 \frac{ \gamma x_*y_*a^2}{(a^2+x_*^2)^2}>r\left(1-\frac{2x_*}{K}\right).  
\end{equation}

\begin{enumerate}
\item[a)] If $2p+n^2-m^2<0$ holds, then the positive equilibrium $(x_*,y_*)$  of system (\ref{eq0})  is locally absolutely stable.

\item[b)] If $2p+n^2-m^2> 2\sqrt{p^2-q^2}$ holds, then there exists a positive integer $k$ such that there are $k$ changes from stability to instability to stability. In other words, the positive equilibrium  $(x_*,y_*)$  of system (\ref{eq0})  is asymptotically stable when
\[\tau\in [0,\tau_{0}^{+})\cup (\tau_{0}^{-} ,\tau_{1}^{+})\cup\cdots\cup(\tau_{k-1}^{-} ,\tau_{k}^{+}), \] 
and unstable when \[\tau\in [\tau_{0}^{+},\tau_{0}^{-})\cup (\tau_{1}^{+},\tau_{1}^{-})\cup \cdots \cup [\tau_{k-1}^{+},\tau_{k-1}^{-})\cup (\tau_{k}^{+} ,+\infty), \]
with $\tau^{\pm}_{j}$ ($j=0,1,2,\cdots$) defined as in (\ref{2.11}).
\end{enumerate}
\end{theorem}
\begin{proof}
        We linealize system (\ref{eq0}) around the equilibrium $(x_*,y_*)$, and grouping the terms of the characteristic equation (\ref{2.6}) we obtain:
\begin{displaymath}
    \lambda^2 +\underbrace{(-a_{11}-a_{22})}_{m}\lambda +\underbrace{(-b_{22})}_{n}\lambda e^{-\lambda \tau}  +\underbrace{a_{11}a_{22}}_{p} + \underbrace{ \left(b_{22}a_{11} -a_{12}b_{21}\right)}_{q}e^{-\lambda \tau}  =0.
\end{displaymath}
where

\begin{displaymath}
 \begin{array}{l}
\displaystyle m = \frac{2\gamma x_*y_* a^2}{(a^2+x_*^2)^2}-r\left(1-\frac{2x_*}{K}\right) +\beta= r\left(1-\frac{x_*}{K}\right)  \left(\frac{2a^2}{a^2+x_*^2}-1\right) +\beta +\frac{rx_*}{K},\\ \\
\displaystyle
 n= \frac{\gamma x_*^2}{a^2+x_*^2} e^{-\frac{y_*}{N}}\left(\frac{y_*}{N}-1\right)= \beta \left(\frac{y_*}{N}-1\right),\\ \\
 \displaystyle p = -\beta\left(r\left(1-\frac{2x_*}{K}\right)-\frac{2\gamma x_*y_* a^2}{(a^2+x_*^2)^2}\right)=-\beta(\beta-m),\\ \\
 \displaystyle q=p-\frac{\gamma x_*^2}{a^2+x_*^2}\frac{2\gamma x_*y_*a^2}{(a^2+x_*^2)}e^{-\frac{y_*}{N}}=p-\beta\frac{2\gamma x_*y_*a^2}{(a^2+x_*^2)}=-\beta r\left(1-\frac{2x_*}{K}\right).
    \end{array}   
\end{displaymath}
Since  $(x_*,y_*)$ is a positive solution of system (\ref{eq_equilibrio}), the following identities hold
\begin{displaymath}
    \frac{\gamma x_* y_*}{a^2+x^2_*}=r\left(1-\frac{x_*}{K}\right),\quad \mbox{and}\quad   \beta=\frac{\gamma x_*^2}{a^2+x^2_*}e^{-\frac{y_*}{N}},
\end{displaymath}
which allow writing $m$, $n$, $p$ and $q$ as in (\ref{mnpq}).

By direct computation, the following inequality can be obtained from (\ref{H2}) 
\begin{displaymath}
      r\left(1-\frac{x_*}{K}\right)  \left(\frac{2a^2}{a^2+x_*^2} -1\right)  +\frac{rx_*}{K} +\beta \frac{y_*}{N}>0,
\end{displaymath}
and consequently $m+n>0$ holds. In a similar way, straightforwardly from (\ref{H3}) we obtain 
$$p+q=-\beta \left( \beta -m +r\left(1-\frac{2x_*}{K}\right)\right)= \beta \left(\frac{2 \gamma x_*y_*a^2}{(a^2+x_*^2)^2}-2r\left(1-\frac{2x_*}{K}\right)\right)>0.$$
Moreover, we note that 
$p-q= \frac{2 \beta  \gamma x_*y_*a^2}{(a^2+x_*^2)^2}>0$, and we conclude that $p^2-q^2>0$. 

Consequently, if $2p+n^2-m^2<0$ holds, we can apply Proposition \ref{prop:C-G} i), and the part a) of the theorem follows. On the other hand, if $2p+n^2-m^2> 2\sqrt{p^2-q^2}$ holds, then Proposition \ref{prop:C-G} iii) implies part b).   
\end{proof}

\subsection{Numeric Stability Analysis}
An algorithm has been developed to calculate and evaluate the equilibria of System \ref{eq0}. This algorithm was implemented as a function in the R language, which takes as input the model parameters and a reference delay value. The diagram in Figure \ref{program} summarizes the functioning of the program. To compute the system's equilibrium points, the `multiroot' function from the `deSolve' library in R was utilized. This function was executed with four different fixed points for $x$ and $y$, from which the iteration starts to find solutions. The function takes as input all the model parameters, initial values of $x$ and $y$, and a delay value. Follow this method is possible to find five different equilibrium points for the system. In cases where critical delays values are identified (according to Theorem 
\ref{Theo4}), the entered value is excluded, and stability behavior is studied based on these new values. The code corresponding to this function and the simulations conducted in this work are available on Posit Cloud\footnote{\url{https://posit.cloud/content/6889846}}. Below are some examples of this function that illustrate the behavior of the system and some of the theorems of this article.

\medskip

\paragraph{\textbf{Absolute stability for positive equilibrium}}
In the case of Figure \ref{exa1}, we identify equilibria denoted as $E_1=(0,0)$, $E_2=(200, 0)$, $E_3=(190.11, 0.94)$, $E_3^*=(9.89, 0.98)$, and $E_3^{\dagger}=(0.06, 0.01)$. We note from our findings that $E_1$ is unstable and since Theorems \ref{Theo2} and \ref{Theo3} are inapplicable, we are unable to determine the stability of the equilibrium point $(200,0)$. On the other hand, $E_3$ is the only positive equilibrium that satisfies the conditions of Theorem \ref{Theo4}. In fact, this result allows us to ensure that $E_3$ is absolutely stable (local asymptotic stability of this equilibrium for all $\tau \geq 0$). This implies that for any delay, the system will converge towards this stable solution, reflecting characteristics of robust stability under varying time delays.

\medskip

\paragraph{\textbf{Changes from stability to instability to stability for positive equilibrium}}
In Figure \ref{exa2}, the equilibria are $E_1=(0,0)$, $E_2=(300, 0)$, $E_3=(204.68,14.45 )$ and $E_3^*=(95.3, 14.45)$. Using our findings on the system, we observe the following conclusions: the origin is unstable, Theorems \ref{Theo2} and \ref{Theo3} are inapplicable, hence the stability of the equilibrium $(300,0)$ cannot be determined, and Theorem \ref{Theo4} indicates that the stability of $E_3=(204.68,14.45)$ is affected by the delay $\tau$. For values of $\tau$ within certain ranges, the system's trajectories will converge to the equilibrium point $E_3$, depending on the initial conditions and the specific value of $\tau$. As $\tau$ increases beyond the respective threshold values$\tau_0+$, this equilibrium point loses stability, leading to possible oscillatory or divergent behavior in the system trajectories. For low delay, $E_3=(204.68,14.45)$ is locally asymptotically stable, but for $\tau = 40.7$ the first change in stability occurs, generating an oscillatory pattern around this point.

\begin{landscape}
\begin{figure*}[h]
\centering
\includegraphics[scale=0.6]{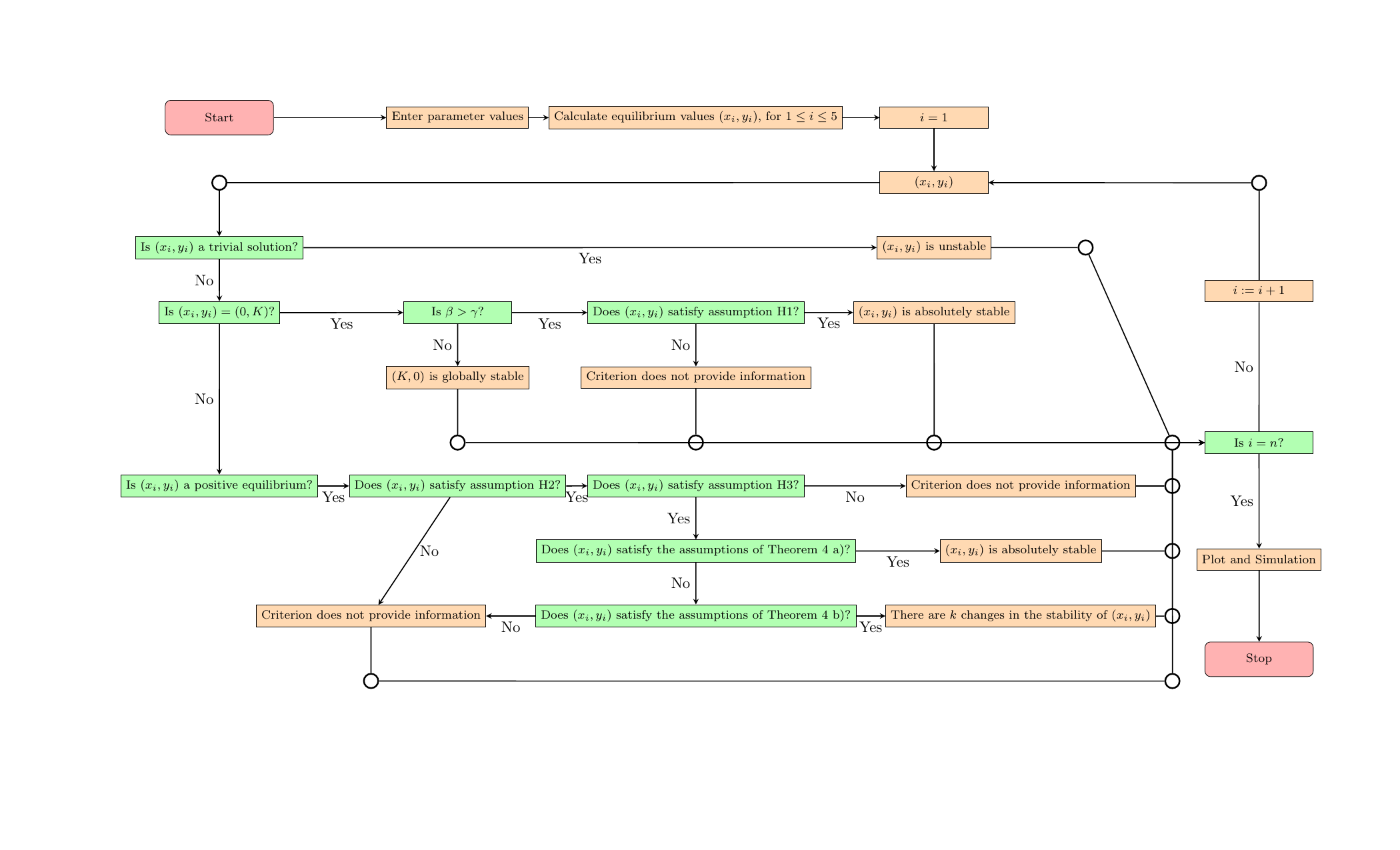}%
\caption{\centering{Stability Analysis Function Diagram}}
\label{program}
\end{figure*}
\end{landscape}

\begin{figure*}[h]
\centering
\fbox{\includegraphics[width=1\textwidth]{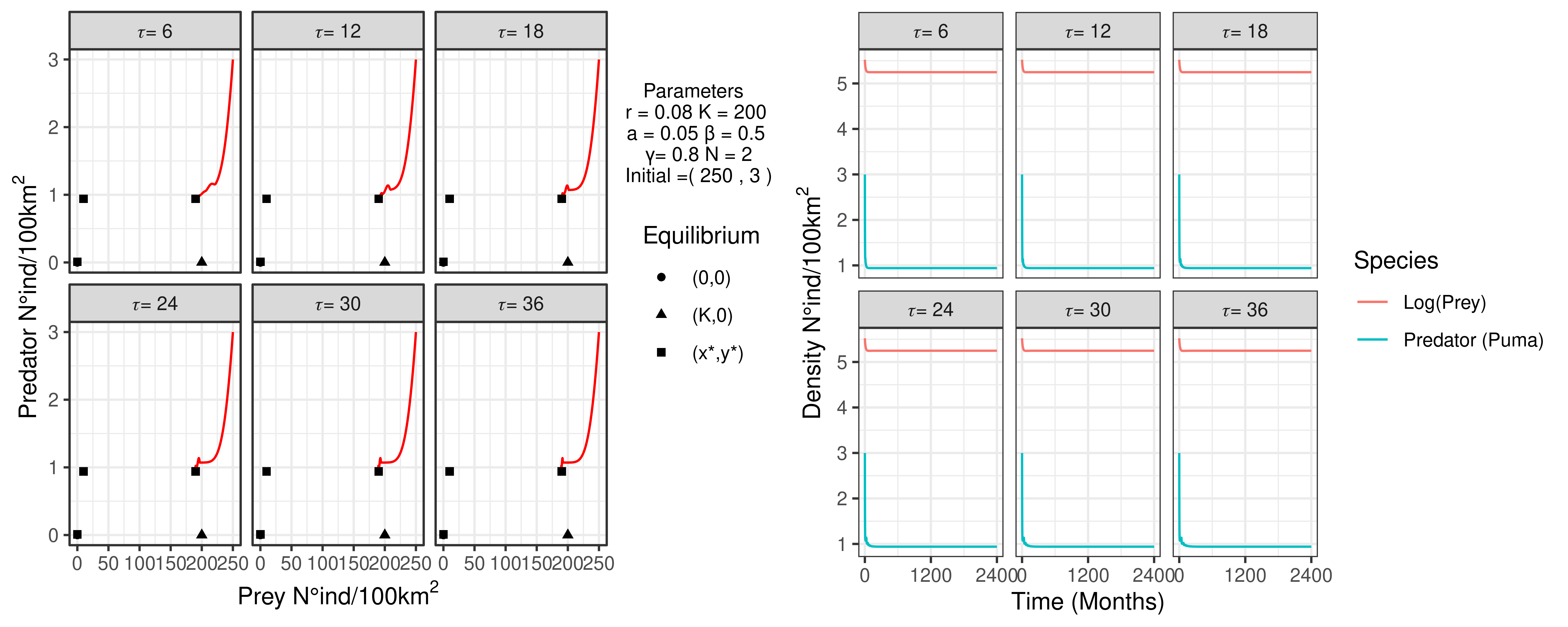}}
\caption{\centering{Absolute Stability}}
\label{exa1}
\end{figure*}

\begin{figure*}[h]
\centering
\fbox{\includegraphics[width=1\textwidth]{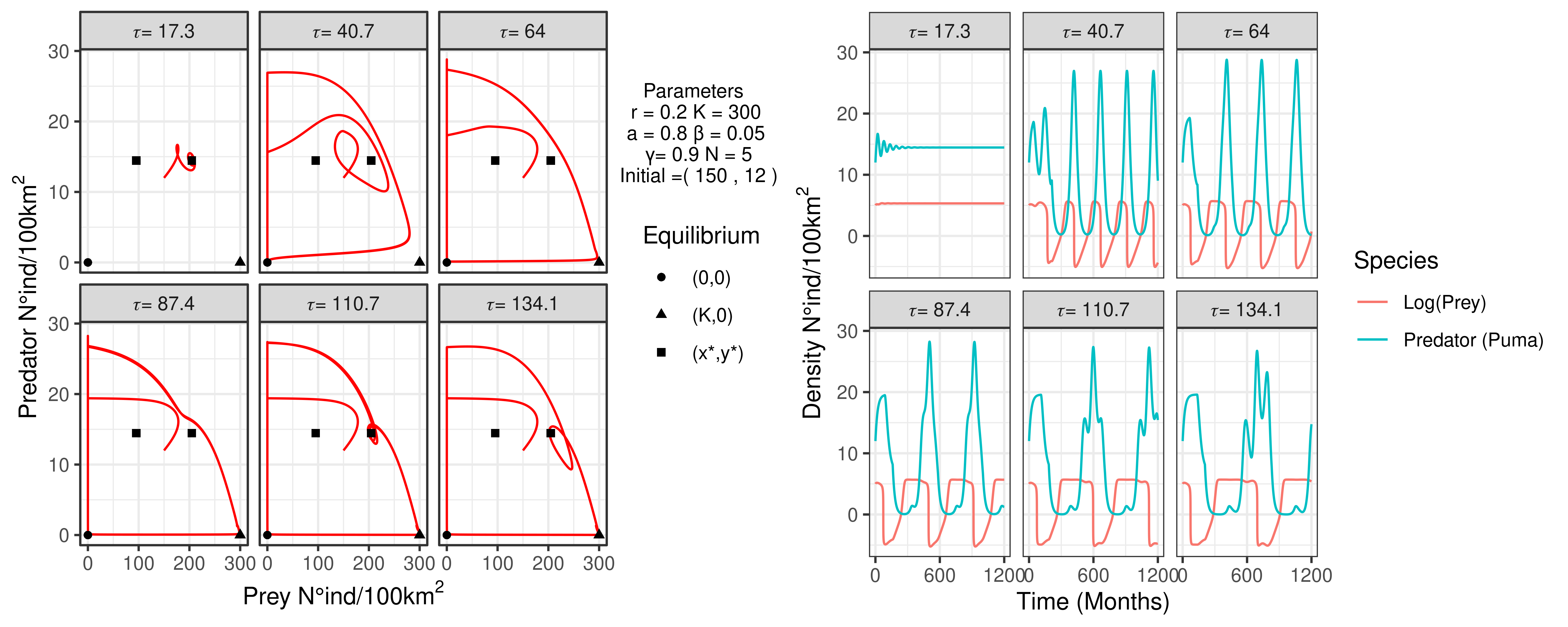}}
\caption{\centering{Changes in Stability Depending on $\tau$}}
\label{exa2}
\end{figure*}
\paragraph{\textbf{Global Stability of $(K,0)$}}

In this example (Figure \ref{exa3}), the equilibria were defined as $E_1=(0,0)$ and $E_2=(50, 0)$, with no positive solutions (condition (\ref{3.8}) is not satisfied). $E_1$ corresponds to an unstable solution, while $E_2$ is globally asymptotically stable for all $\tau > 0$. Indeed, in this example, the condition established in Theorem \ref{Theo3} is met since $\beta >  \gamma $, so consequently the convergence towards $E_2=(50,0)$ is appreciated.

\begin{figure*}[h]
\centering
\fbox{\includegraphics[width=1\textwidth]{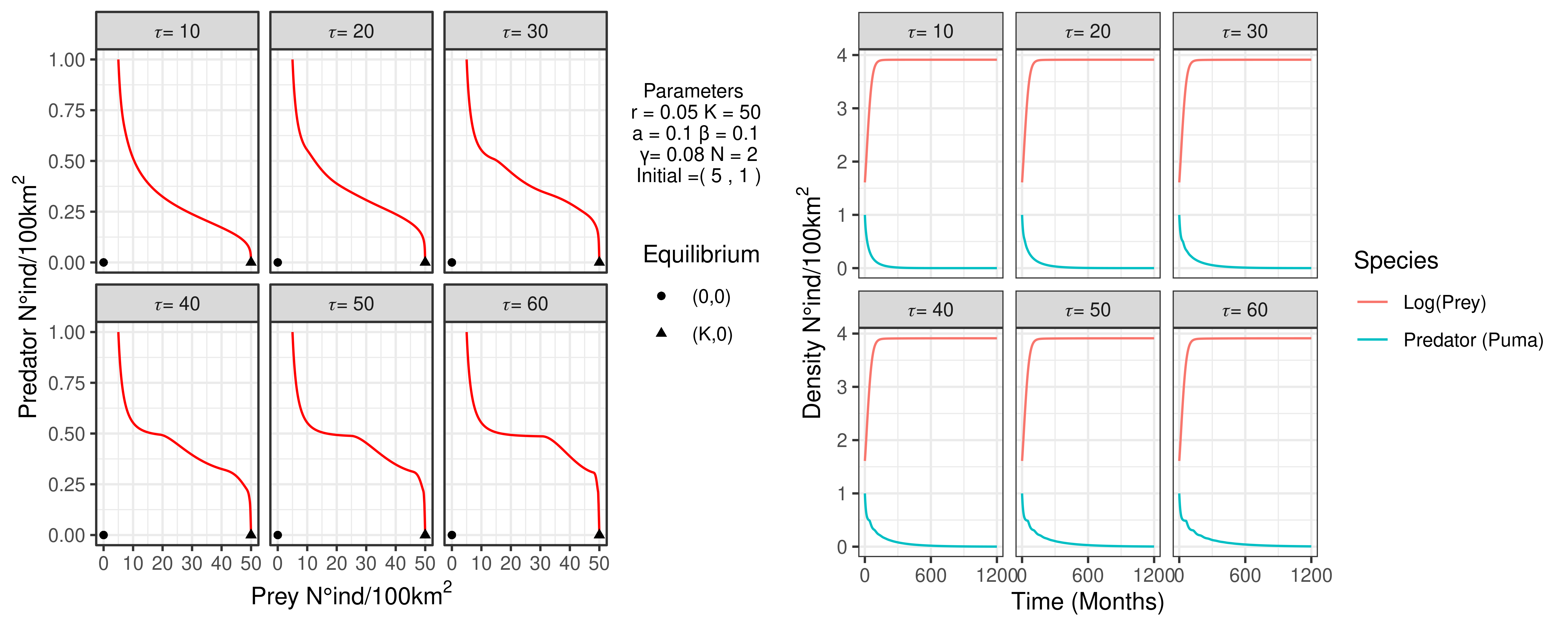}}
\caption{\centering{Stable Solution $(K,0)$ for all $\tau > 0$}}
\label{exa3}
\end{figure*}

It is important to highlight that the stability analysis of a system of this type, which has multiple interacting equilibrium points, is not a trivial exercise. Numerical simulations attempt to exemplify and put into practice the theorems proved in this article. Another technique that can be useful to complement this analysis are stability charts, where it is possible to show the regions in the plane $(\tau, c)$ (where $c$ is some parameter of the model) where the system is stable. Examples of applications of this technique are the works of \citet{kalmar2009stability}, \citet{insperger2011stability}, and \citet{jin2020stability}.\\

\subsection{Model simulations}
For the simulation of the model (\ref{eq0}), the parameter values presented in Table \ref{tab1} were considered. With these combinations of parameters, a total of 216 simulations were performed for the model (\ref{eq0}), and the results were grouped according to the behavior exhibited by the predator and prey populations. The categories in which the results were classified are as follows: Populations in positive equilibrium (72.2\%), predator extinction (26.6\%), and oscillating populations (1.4\%). The prey density graphs were plotted on a logarithmic scale, since their values differ by several orders of magnitude from those of the predators, making it difficult to visualize them accurately. \footnote{Code available on Posit Cloud \url{https://posit.cloud/content/6889846}}
\begin{table*}[ht]
\footnotesize
\caption{\centering{Model simulation parameters}}
    \centering
    \begin{tabular}{llr}
        \toprule
        Parameters & Description & Simulated values\\
        \midrule
        $r$ & Prey reproduction rate & 0.05, 0.1, 0.2 \\
        $K$	& Carrying capacity &	200, 500 \\
        $a$ & Average predator saturation constant & 0.1, 0.5, 0.8\\
        $\gamma$ & Maximum per capita predator consumption rate & 0.1, 0.5, 0.8\\
        $\beta$ &	Predator mortality rate & 0.05, 0.1  \\
        $N$ & Optimal reproduction size for the predator population & 1, 2 \\
        $\tau$ & Period that includes gestation and the time it takes to & 27  \\
             & adulthood for puma & \\
      \bottomrule
    \end{tabular}
    \label{tab2}
\end{table*}
\paragraph{\textbf{Predator extinction}}
This situation occurs when predators have a mortality rate ($\beta$) higher than the largest possible value of the functional response ($\gamma \frac{K^2}{a^2+K^2}$) (Figure \ref{fig4}). In this case, predators have no significant impact on prey, as their per capita consumption parameter ($\gamma$) is low. This leads to the prey population continuing to grow exponentially at their reproduction rate ($r$) until it reaches carrying capacity ($K$), while the predator density approaches zero. Varying the optimal reproduction size parameter for predators ($N$) would not change the observed trend. It is important to note, as pointed out in \citet{rumiz2010roles}, that the role of carnivores is to control the populations of certain species, and as shown in the simulation, an increase in Puma mortality leads to a rapid growth of the prey population. The literature mentions that the loss of this biological control can result in pest cycles that can cause local extinctions of different plant and animal species \citep{rios2009puma}.

\medskip

\begin{figure*}[ht]
\centering
\fbox{\includegraphics[width=1\textwidth]{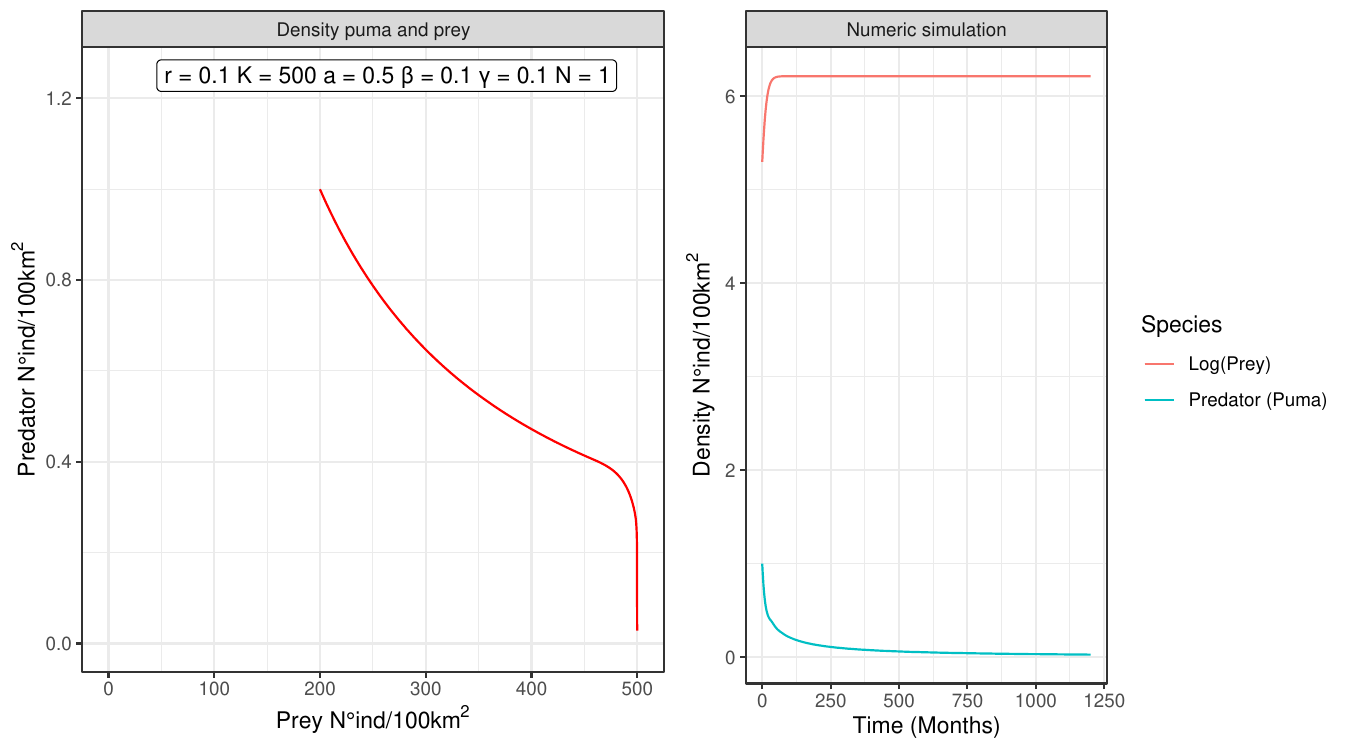}}
\caption{\centering{Predator extinction simulation}}
\label{fig4}
\end{figure*}

\paragraph{\textbf{Populations tend towards a positive equilibrium}}
In this case, the predator mortality rate ($\beta$) is lower than the maximum possible value of the functional response, which, according to Theorem 3, guarantees the existence of positive equilibria. In addition, the saturation constant ($a$) exceeds the maximum per capita predator consumption rate ($\gamma$), leading to the predator population regulating itself as it consumes the prey, converging towards a fixed value. Meanwhile, although the prey population exhibits oscillations in its density, its high reproduction rate ($r$) enables it to remain close to the carrying capacity ($K$) (Figure \ref{fig5}). On the other hand, the low predator mortality rate ($\beta$) allows them to maintain a density higher than that defined by the parameter ($N$).\\
\begin{figure*}[ht]
\centering
\fbox{\includegraphics[width=1\textwidth]{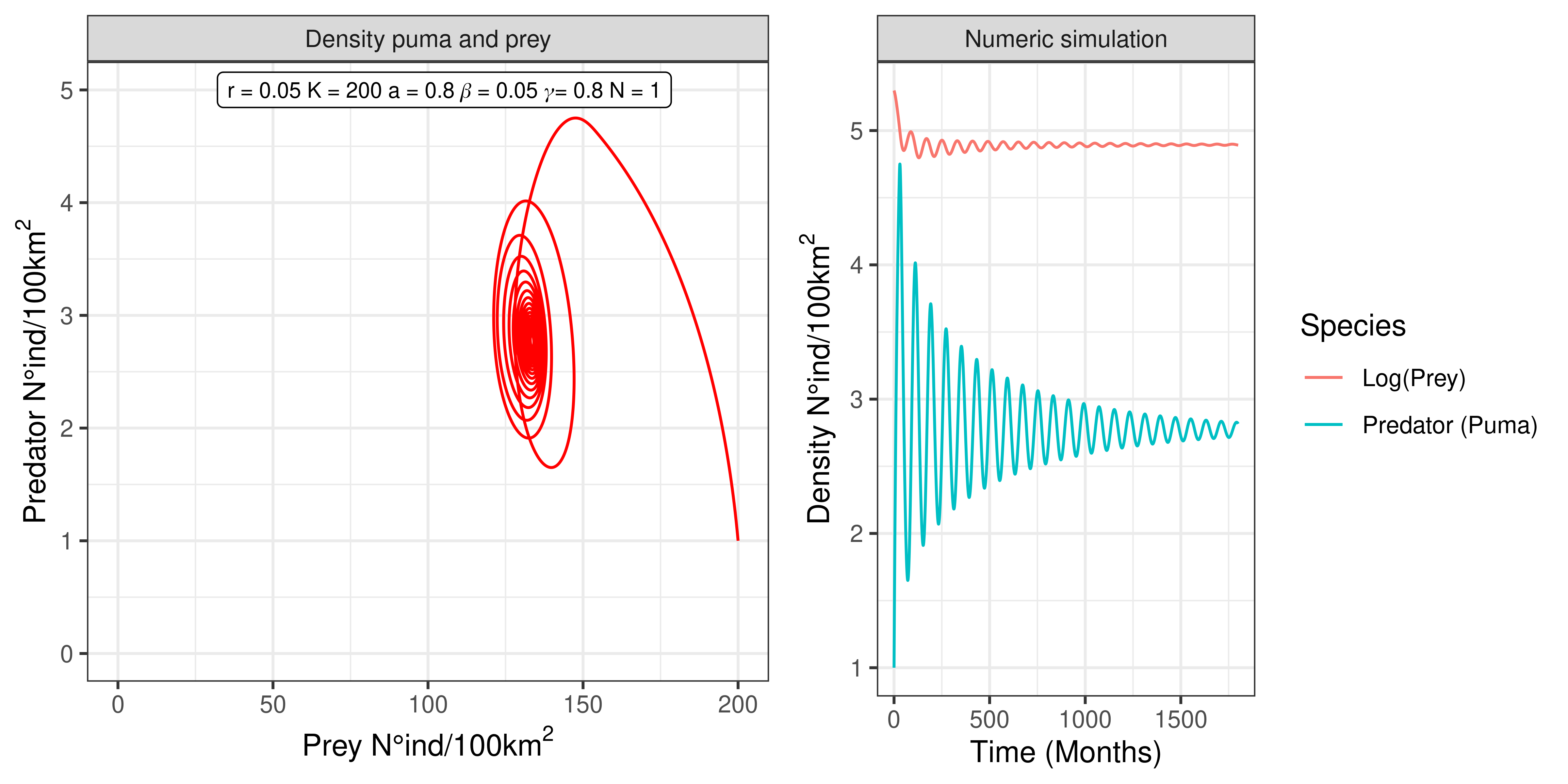}}
\caption{\centering{Simulation of populations in positive equilibrium}}
\label{fig5}
\end{figure*}
\paragraph{\textbf{Oscillating population}}
In this case, each curve exhibits temporally shifted cyclical peaks for both the predator and the prey (Feedback) (Figure \ref{fig6}). Once the puma reaches its maximum density, the prey population declines to such an extent that food begins to become scarce, consequently resulting in a decrease in the predator population, almost to the point of extinction. Subsequently, the number of prey stabilizes and begins to increase, initiating a new cycle. In all simulations that produced this type of result, the prey reproduction rate ($r$) was low and the carrying capacity was established at 200 prey per 100 km², while the predator mortality rate ($\beta$), the maximum per capita predator consumption rate ($\gamma$) and the optimal reproduction size ($N$) were established at their maximum evaluated values. The average saturation constant of predators varied in all cases and did not appear to be influential in this type of outcome. Due to high predator mortality, low birth rate, and limited carrying capacity of prey, neither population reaches a stable state over time. This, combined with the high hunting efficiency of the predators, generates the oscillatory behavior between predator and prey densities. It should be noted that the predator density follows the typical pattern associated with a Nicholson-type model \citep{gurney1980nicholson}.
\begin{figure*}[ht]
\centering
\fbox{\includegraphics[width=1\textwidth]{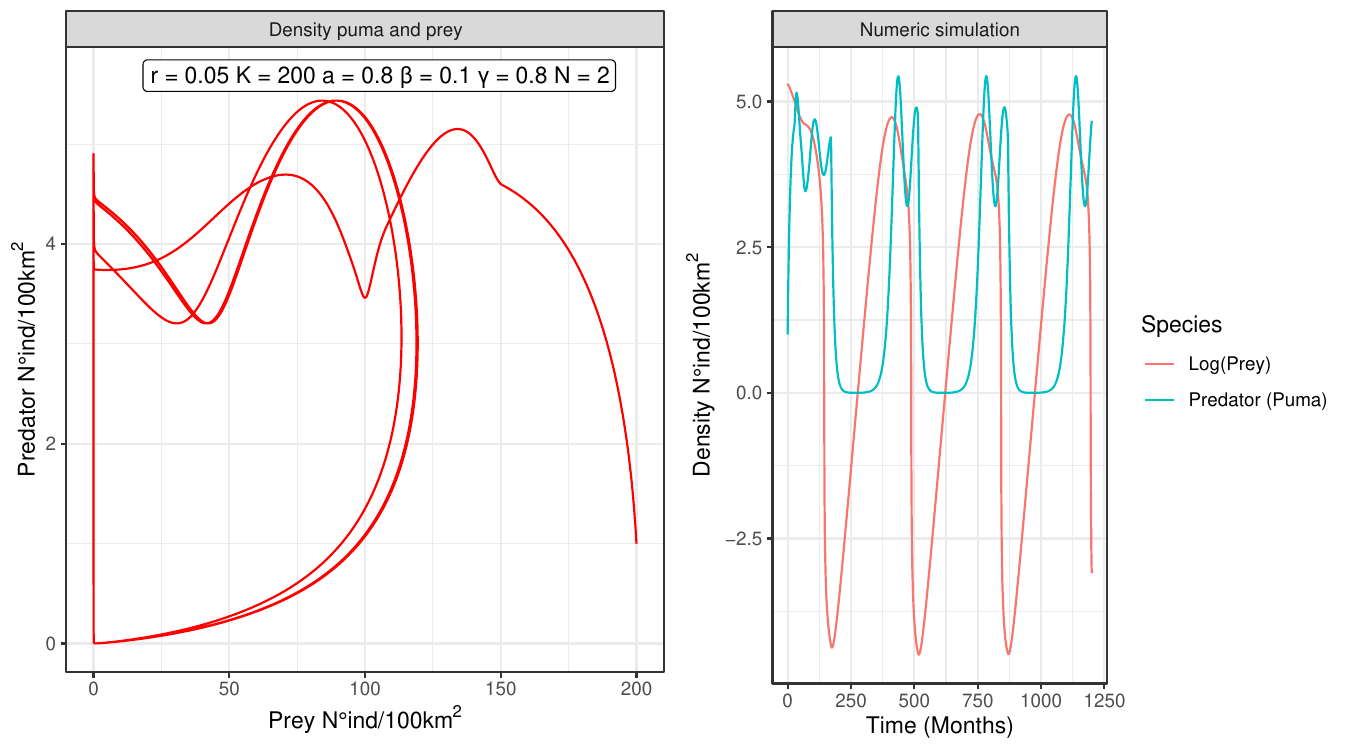}}
\caption{\centering{Simulation of oscillatory densities}}
\label{fig6}
\end{figure*}

\subsection{Simulation of potential scenarios}
\noindent
To complement the conducted analyses, the evaluation of two potential scenarios associated with real disruptions and threats to the ecosystems inhabited by the Puma was included. The proposed scenarios were: 
Reduction of the prey population, and predator population. In each case, systematic elimination of a percentage of prey or Pumas is carried out, as appropriate, when they exceed a certain threshold. 

To conduct the analyses of the scenarios, we followed the numerical procedure outlined in subsection 7.6 of the book by \citet{soetaert2010solving}. We evaluated the scenarios where the populations tend towards a positive equilibrium and the scenarios with oscillatory behavior of the populations.
\subsubsection{Preys removal}
A simulation was conducted where 50\% of the prey population was systematically removed when it reached a certain threshold value. For the case of the positive equilibrium, prey were removed when they reached a value of 150 individuals per km², while for the case of the oscillatory equilibrium, this occurred when 20 individuals per km² were reached. The difference in these values is related to the values that the prey population reaches in the baseline cases (Figure \ref{fig5} and Figure \ref{fig6}).
\begin{figure*}[ht]
\centering
\fbox{\includegraphics[width=1\textwidth]{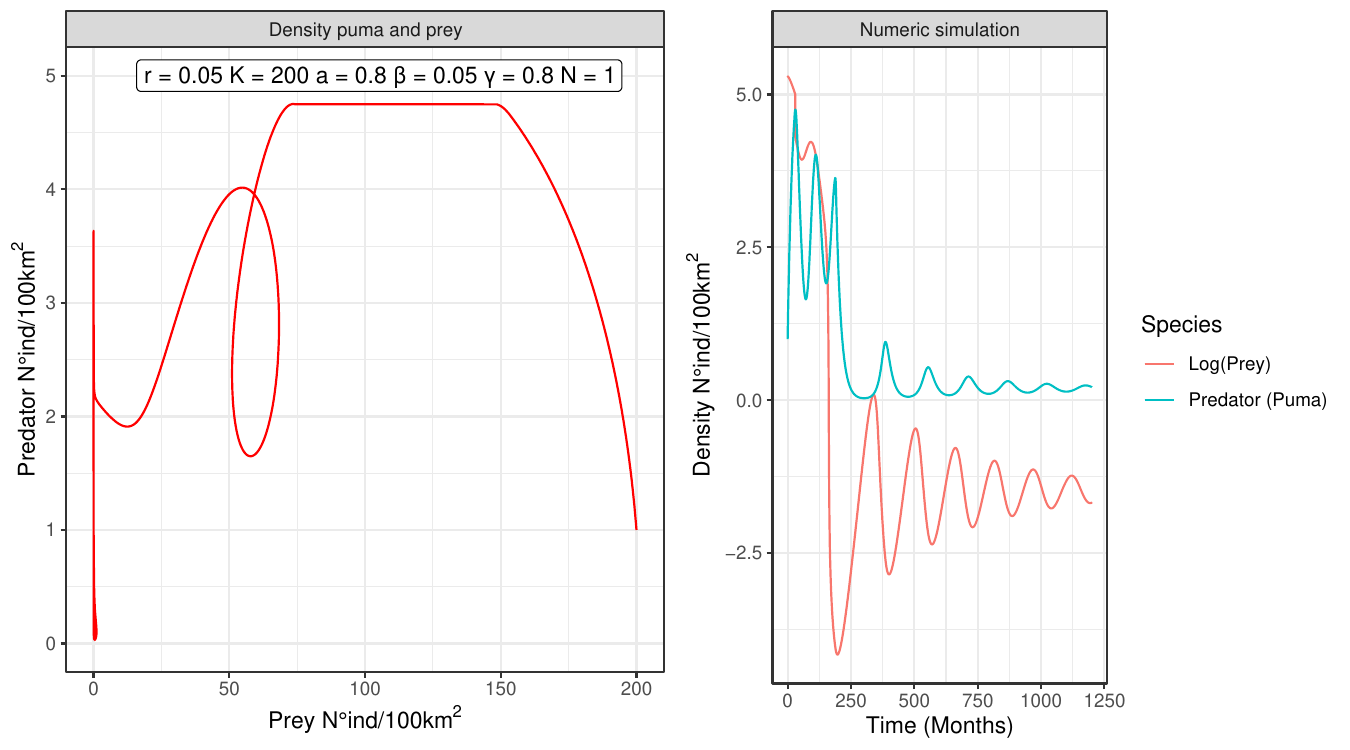}}
\caption{\centering{Case of positive equilibrium with prey removal}}
\label{fig7}
\end{figure*}
In the case of the positive equilibrium, the removal of prey resulted in a significant decrease in their population in the following months. This reduction caused the predator population to tend towards extinction. This type of solution differs drastically from the solution depicted in Figure \ref{fig5}. The prey density falls below the minimum threshold required to sustain the predator population, which is an expected consequence of this perturbation. The general solution of the system implies that both predators and prey dramatically decrease their density, tending to values of the order of 0.1 individuals per km². 
\begin{figure*}[h]
\centering
\fbox{\includegraphics[width=1\textwidth]{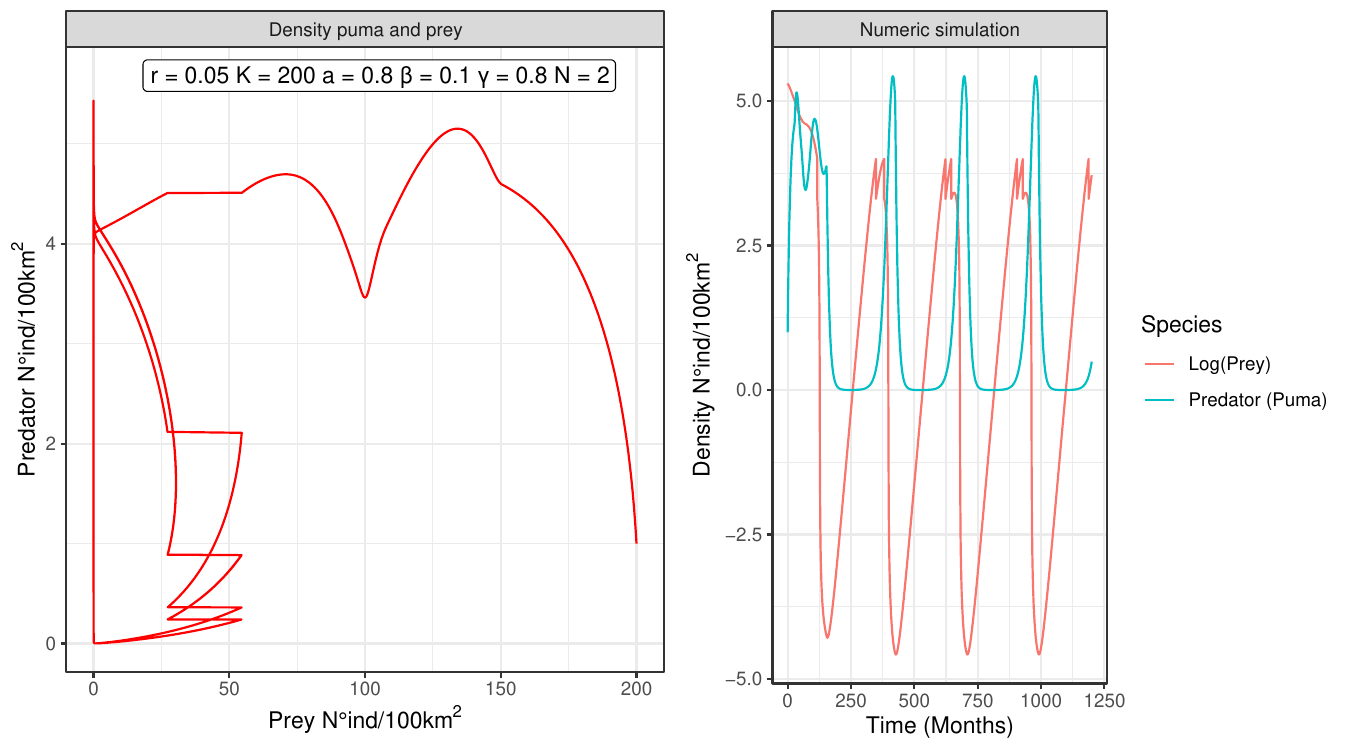}}
\caption{\centering{Oscillatory case with prey removal.}}
\label{fig8}
\end{figure*}
The removal of prey in the oscillatory scenario (Figure \ref{fig8}) leads to changes in the dynamics of the predators, as the characteristic behavior of a population represented by a Nicholson-type model is lost. Despite this, the overall densities of predators and prey still follow a similar pattern to the undisturbed system, where the puma population reaches its peak and then declines rapidly. We can observe that the interactions between the densities of predators and prey persist, but the systematic removal of a fraction of the prey reduces the duration of periods of higher puma abundance.
\subsubsection{Predator removal}
Two scenarios were investigated that involved the removal of 50\% of the predator population. In the first case, known as positive equilibrium, a threshold value of 3 individuals per km² was defined to achieve a reduction of half of the population. In the second case, referred to as the oscillatory case, a threshold of 4 individuals per km² was established.
\begin{figure*}[ht]
\centering
\fbox{\includegraphics[width=1\textwidth]{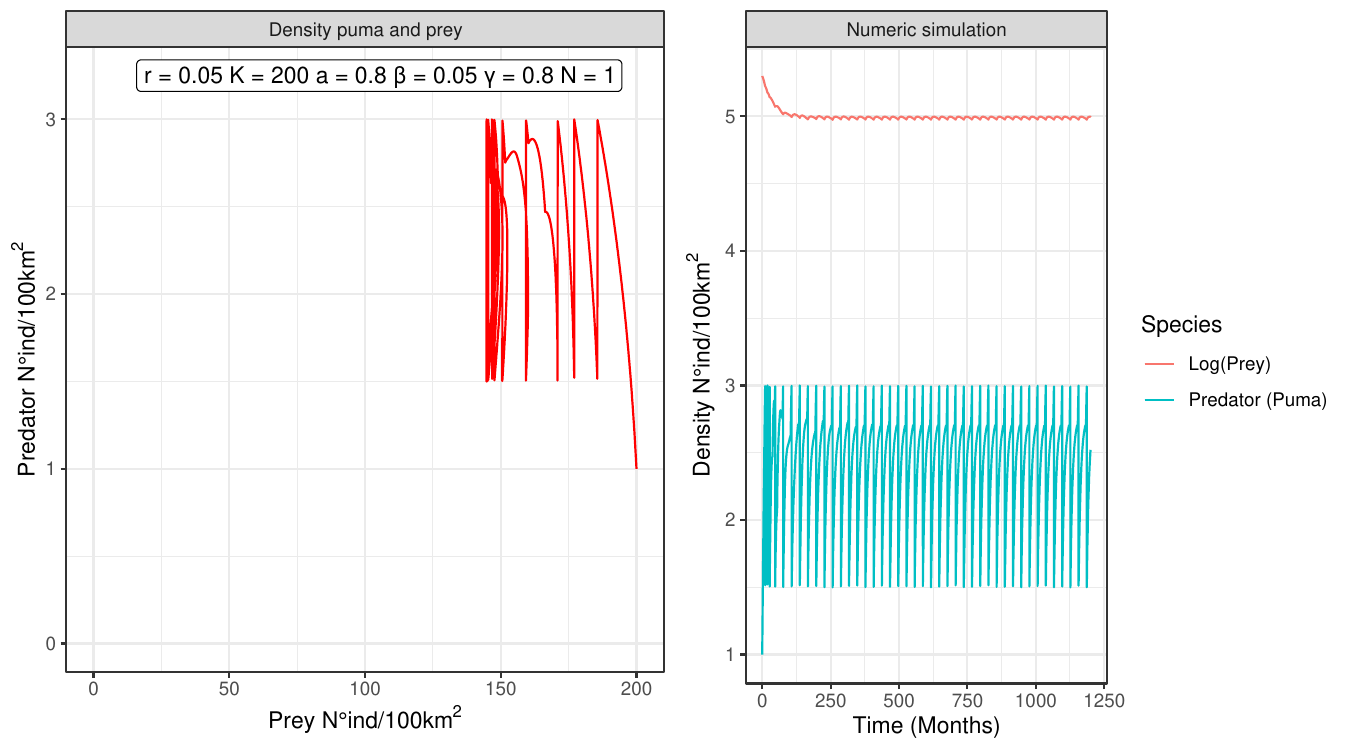}}
\caption{\centering{Positive equilibrium case with predator removal}}
\label{fig9}
\end{figure*}
The reduction in puma density at the positive equilibrium significantly affected the solution behavior of the system (Figure \ref{fig6}). The feedback interactions between predators and prey were completely disrupted. As a result, the prey population decreased to a steady state value, remaining constant over time. However, the puma population exhibited oscillations between the predetermined maximum density and 50
\begin{figure*}[h]
\centering
\fbox{\includegraphics[width=1\textwidth]{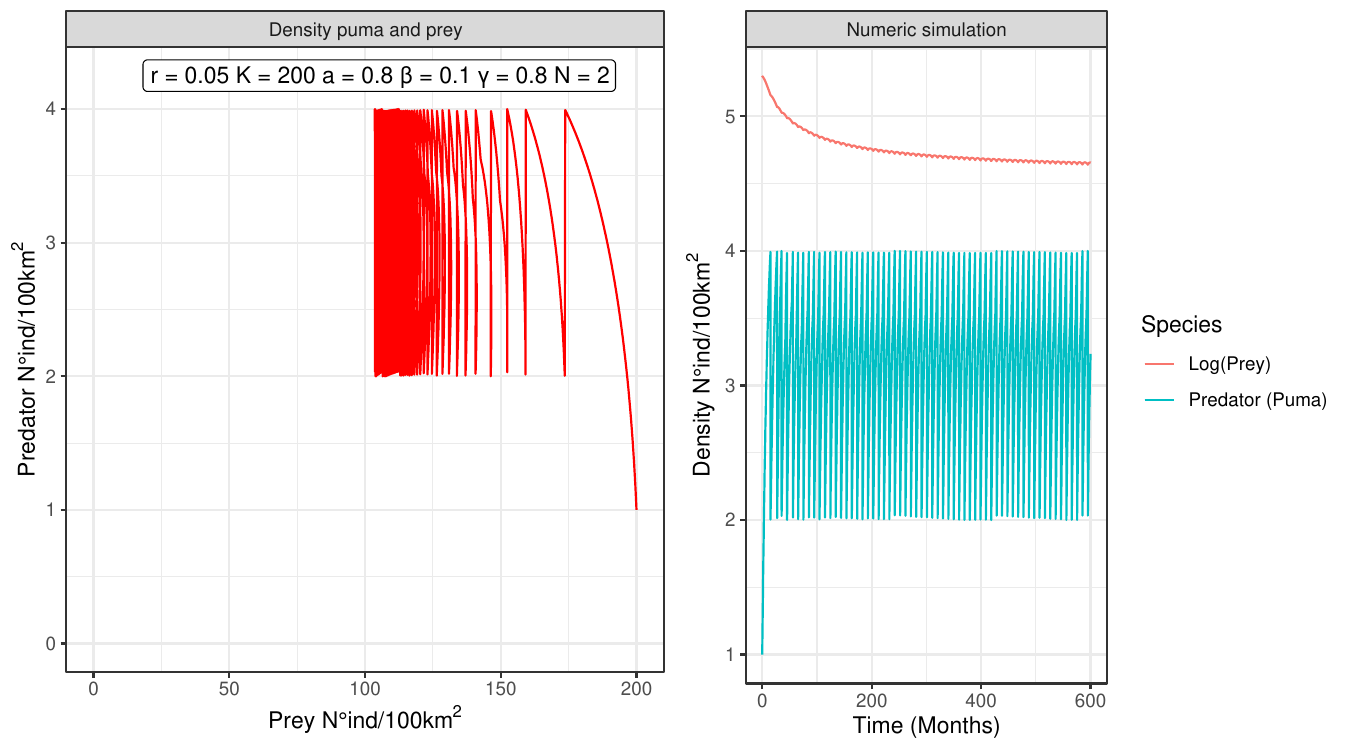}}
\caption{\centering{Oscilatory equilibrium case with predator removal}}
\label{fig10}
\end{figure*}

The decrease in puma density in the oscillatory case leads to a disruption in the system's behavior (Figure \ref{fig10}). This outcome is reminiscent of the effect observed when removing predators in a positive equilibrium scenario, although with the distinction that the oscillations in puma density adapt to the values defined specifically for this case.

In comparison to the perturbation caused by reducing the prey population, the decline in predators can have a significantly more pronounced impact on the system's dynamics. This behavioral shift is analogous to the findings reported in the study by \citet{barman2022role}, which demonstrated that implementing separate harvesting strategies for prey and predators can induce a change in the system's stability. Furthermore, the work of \citet{martin2001predator} extensively explores predator-prey systems with delays, emphasizing that the magnitude of species removal affects the stability of the system's solutions. They identified a critical harvesting threshold beyond which the stability of the solutions undergoes modification. It would be of great interest to determine the lower threshold values at which the behavior of the model proposed in this study is altered, as this information would be crucial in establishing management criteria for puma densities and their potential prey.



\section{Conclusions}

This study explores various aspects of a model that describes the behavior of a predator-prey system. The model, based on delay differential equations, incorporates logistic growth for the prey, a type III functional response, and a maturation period for the population where intraspecific competition, modeled by a Ricker-type function, occurs for the pumas. The model ensures that an increase in predator density decreases the per capita growth rate of the prey, while prey density exerts a positive effect on the per capita growth rate of the predator. Additionally, the puma reproduction rate has a maximum value that decreases with increasing population density or in the absence of sufficient prey. These characteristics align with the attributes identified by \citet{berryman1995credible} as fundamental for the credibility of such models. To our knowledge, this is the first model that integrates the aforementioned aspects for the species Puma concolor.

The analysis demonstrates the existence of a unique solution for the proposed model, which remains non-negative for non-negative initial conditions and is well-defined for all \( t \) greater than zero. The research involves a thorough mathematical analysis of the stability of the system's equilibria, using both Liapunov functions and LaSalle's invariance principle, as well as the distribution of the roots of the characteristic equation of the linearized system. In fact, the equilibrium points of the system are identified, and their stability is examined. The extinction scenario, where both predators and prey disappear, is determined to be a locally unstable equilibrium. In contrast, conditions are established that guarantee absolute stability for the equilibrium point where the prey reach their carrying capacity while the puma population goes extinct. For the positive equilibrium, criteria are established to evaluate its conditional stability, depending on the value of the delay and the model parameters. Furthermore, an R routine was implemented to systematize the theoretical stability analysis, making it numerically applicable. The results obtained are consistent with those of other similar studies \citep{kar2005stability, liu2010analysis, sun2018stability}.

In addition, numerical simulations with different parameter values were performed to investigate the effects of systematically removing a percentage of predators or prey. Although these simulations provide insight into the system's behavior, they only offer an approximation and do not fully capture its complexity. Therefore, it is necessary to continue collecting data to refine the model parameters or introduce new ones. In particular, the numerical analysis considered a fixed delay value associated with the gestation period and the time it takes for pumas to reach sexual maturity.

Simulations with varying parameter values reveal three possible scenarios for the model solutions. The first scenario involves the extinction of the predator, while the second leads to the stabilization of both predator and prey populations. The third scenario shows a constant feedback dynamic between predator and prey densities. Specifically, increasing the mortality rate of the puma leads to rapid growth in the prey population, which can trigger pest outbreaks and pose a significant threat to ecosystem biodiversity.

Changes in the densities of prey and predator populations induce changes in the behavior of the system. Systematic removal of a fraction of the prey, in the case where the model has a locally stable positive equilibrium, causes a significant shift in the system dynamics, now stabilizing at a point with a lower number of predators and potentially leading to their extinction. In the oscillatory case, the reduction of prey shortens the duration of periods with the highest density of the puma. On the other hand, systematic removal of a fraction of the predators, both in the case where the model has a locally stable positive equilibrium and where the system has an oscillatory solution, makes the dynamics of both populations oscillatory without risking the extinction of either.

\section*{Acknowledgments} 
The authors would like to thank Professor G. Robledo for several conversations that contributed to the development of this manuscript.

\bibliography{main}
\bibliographystyle{apalike}

\end{document}